\documentclass[twocolumn]{autart}    % Enable this line and disable the 

\usepackage{graphicx}          % Include this line if your 

\usepackage{amsmath, amssymb}
\usepackage{xcolor}
\usepackage{graphicx}
\usepackage{float}
\usepackage{placeins}

\usepackage{algorithm}     % 算法浮动体
\usepackage{algorithmicx}  % 算法伪代码（推荐，兼容性好）
\usepackage[noend]{algpseudocode}  % 或者使用这个
\newtheorem{problem}{Problem}
\newtheorem{lemma}{Lemma}
\newtheorem{remark}{Remark}
\newtheorem{theorem}{Theorem}

\begin{document}
	
	\begin{frontmatter}
		
		\title{Distributed Algorithm for the Global Optimal Controller of Nonlinear Multi-Agent Systems} 
		
		\thanks[footnoteinfo]{This work was supported by the National Natural Science Foundation of
			China under Grants 62573262, 62503289 and the Natural Science Foundation
			of Shandong Province under Grant ZR2021JQ24.\\
			Corresponding author Juanjuan Xu. 
		}
		
		\author[CSE]{Ruixue Li}\ead{ruixueli777@163.com},   
		\author[CSE]{Wenjing Yang}\ead{yangwenjing1024@163.com},   
		\author[CST]{Zhaorong Zhang}\ead{zhangzr@sdu.edu.cn},             
		\author[HK]{Xun Li}\ead{li.xun@polyu.edu.hk},
		\author[CSE]{Juanjuan Xu}\ead{juanjuanxu@sdu.edu.cn}
		
		\address[CSE]{School of Control Science and Engineering, Shandong University, Jinan, China}                                            
		\address[CST]{School of Computer Science and Technology, Shandong University, Qingdao, China}            
		\address[HK]{Department of Applied Mathematics, The Hong Kong Polytechnic University, Hong Kong, China}     
		
		                             % keyword list or with the 
		% help of the Automatica 
		% keyword wizard
		
		\begin{abstract} % Abstract of not more than 200 words.
			In this paper, we investigate the distributed optimal control problem for a kind of nonlinear multi-agent systems.
			In particular, both the state and the system dynamic structures of each agent are private and can only be shared among communicating agents. This type of information structure is inevitable in fields such as collaborative control for industrial confidentiality, and renders traditional distributed control methods using all systems' dynamic structures ineffective. 
			The primary contribution is the proposal of a distributed algorithm for the global optimal controller under such practical information structure via distributed approximation of the Hamilton-Jacobi-Bellman equation.
			Practical numerical simulation demonstrates the effectiveness of the proposed algorithm.
		\end{abstract}
		
	\end{frontmatter}

	\section{Introduction}
	Owing to its unique ability to accurately characterize and handle the pervasive nonlinear dynamic characteristics in practical engineering systems \cite{1}, nonlinear optimal control plays an irreplaceable role in a host of critical engineering fields. Typical applications include the high-precision trajectory tracking of aerospace vehicles, autonomous navigation of intelligent mobile platforms, and optimal dispatching of renewable energy grids \cite{2MAS}.

	Unlike linear optimal control problems that can be solved to yield rigorous analytical solutions via the Riccati equation, nonlinear optimal control faces severe theoretical and computational challenges due to the requirement of solving the Hamilton-Jacobi-Bellman (HJB) equation \cite{BrysonHo1975}, which rarely has an analytical solution \cite{VI}. 	
	In light of this, extensive research efforts have been devoted to developing  strategies for nonlinear optimal control problems.
	\cite{non4}
	proposed a numerically efficient discrete-time algorithm for nonlinear optimal control with quadratic criteria.
	\cite{non2}
	investigated the fundamentals of the control parameterization method and reviewed its various applications to the non-standard optimal control problems.
	\cite{non3}
	addressed the discrete-time infinite-horizon optimal control problems with a terminal state set constraint.
	\cite{non2019}
	addressed nonlinear optimal control problems using a data-driven framework that combines indirect methods with k-nearest neighbors retrieval and sensitivity analysis.
	\cite{non1}
	studied a continuous-time value iteration method that will be applied to address the adaptive or nonadaptive optimal control problems for continuous-time systems described by differential equations. 
	\cite{non6}
	investigated finite-horizon constrained robust optimal control problem  affected by norm-bounded disturbances.

	Despite notable advances in nonlinear optimal control research, the vast majority of existing solution frameworks adopt a centralized architecture for optimal controller design, which relies on global system information and thus inherently suffers from single-point failure vulnerability, and communication bandwidth bottleneck \cite{6bandwidth,7risk}. These intrinsic characteristics severely impede the scalability of centralized methods in large scale engineering systems and render them infeasible for practical scenarios with stringent communication constraints and information privacy requirements.

	On the contrary, distributed control approaches possess distinct advantages, including strong adaptability, relaxed system requirements, high robustness, and flexible scalability \cite{disadvant}, which render them particularly well-suited for systems constrained by topological structures and large-scale network subsystems \cite{DisMAS}. For these compelling reasons, distributed methods have emerged as an increasingly desirable alternative for addressing complex control problems \cite{moreresearch}.
	For instance, 
	\cite{dis1} presented a linear quadratic (LQ) optimal control-based approach to solve the multi-agent systems consensus problem by minimizing the weighted state error with neighboring agents.
	\cite{dis2} investigated the optimal consensus control problem for continuous-time multi-agent systems via a distributed algorithm based on the alternating direction method of multipliers.	
	\cite{dis3} studied an inverse optimality approach combined with partial stability to address cooperative consensus and pinning control for agents with identical linear time-invariant dynamics and fixed directed communication topologies.
	\cite{dis4} investigated the structural properties of infinite-horizon LQ optimal control problems by analyzing the spatial structure of solutions to the associated operator Lyapunov and Riccati equations.
	\cite{Dissolving}
	addressed the LQ optimal control problem for continuous-time systems with terminal state constraints.

	However, all the above studies focus on the distributed optimal control of linear systems. Most practical multi-agent systems exhibit inherent nonlinear characteristics, and linearization approximations often introduce unignorable errors that degrade control performance or even lead to system instability under complex operating conditions. Therefore, the research on distributed optimal control for nonlinear multi-agent systems is of great theoretical and engineering value.

	In this paper, we focus on the distributed optimal control problem for nonlinear multi-agent systems, where each agent is governed by affine dynamics and subject to the practical significant information structure, under which each agent can only utilize its own and its communication neighbors’ states and system dynamic structure.
	The main contribution of this paper lies in the design of a distributed algorithm for the global optimal controller via distributed approximation of the HJB equation. 
	Specifically, under the information structure constraint, the proposed algorithm enables each agent to construct the distributed HJB equation using only the accessible information within its own information structure. 
	Practical numerical simulation validates the effectiveness of
	the proposed algorithm.

	The remainder of this paper is outlined as follows. Section \ref{sec2} gives preliminary knowledge and formulates the problem to be addressed. 
	Section \ref{sec3} presents the global optimal solution to the optimal control problem using global information. Section \ref{sec4} introduces the distributed algorithm for approximating the global optimum under the considered information structure.
	Practical numerical simulation is given in Section \ref{sec5} to verify the effectiveness of the proposed algorithm. 
	Finally, detailed proofs are relegated to the Appendix for the readability of the main text.
	
	The notation used in this paper is listed below. $\mathbb{R}^n$ denotes the set of $n$-dimensional real vectors. $\mathbb{R}^{n \times m}$ denotes the set of $n\times m$-dimensional real matrices. $Z'$ represents the transposed matrix of matrix $Z$. $\left \| \cdot \right \|$ is the vector norm on the Euclidean space. $\nabla_xV $ denotes the gradient of the scalar field 
	$V$ with respect to the variable $x$.
	The communication topology among $N$ agents is modeled by an undirected graph $\mathcal{G} = (\mathcal{V}, \mathcal{E}, \mathcal{A})$, where $\mathcal{V} = \{1, 2, \dots, N\}$ represents the agent set, $\mathcal{E} \subseteq \mathcal{V} \times \mathcal{V}$ denotes the communication links, and $\mathcal{A} = (a_{ij}) \in \mathbb{R}^{N \times N}$ is the weighted adjacency matrix with $a_{ii} = 0$ and $a_{ij} > 0$ if $(i,j) \in \mathcal{E}$ where $i \neq j$.
	 The degree of agent $i$ is given by $d_i = \sum_{j=1}^N a_{ij}$, with the degree matrix $\mathcal{D} = \text{diag}\{d_1, \dots, d_N\}$ and Laplacian matrix $\mathcal{L} = \mathcal{D} - \mathcal{A}$ characterizing the graph connectivity. Within this framework, each agent $i$ communicates exclusively with its neighbors defined as $\mathcal{N}_i = \{ j \in \mathcal{V} \mid (i, j) \in \mathcal{E} \}$.

	\section{Problem formulation }	\label{sec2}

	In this section, we first introduce the class of affine nonlinear multi-agent systems under investigation. Then, accounting for the private nature of information interaction among agents, we define a practically meaningful information structure to precisely characterize the accessible information for distributed optimal controller design, based on which we formally formulate the problem to be addressed.

	Consider a kind of multi-agent systems consisting of $N$ affine nonlinear agents, in which the dynamics of the $i^{th}$ agent is described as
	\begin{align} \label{system_agent}
		\dot{x}_i(t)=f_i(x_i(t))+g_i(x_i(t))u_i(t),
	\end{align}
	with the initial state $x_i(0) = x_{i0}$,
	where $x_i(t) \in \mathbb{R}^n $,  $u_i(t) \in \mathbb{R}^m $ are the state and control input of agent $i$, respectively, $f_i$, $g_i$ are continuously differentiable functions, and $ t \in [0, T]$ represents the continuous-time variable.

Such affine nonlinear multi-agent systems have been widely applied in practical engineering scenarios \cite{qUAV,multi-joint manipulator,ASV}. 
Notably, the dynamic properties of each agent $i$ are determined by the local system dynamic structures $f_i$ and $g_i$. 
%In consideration of agent heterogeneity as well as the demand for security and privacy preservation,
In consideration of the demand for security and privacy preservation,
 the state $x_i$, the control input $u_i$ and system dynamic structures $f_i$, $g_i$ are all regarded as private information of agent $i$.

Moreover, to capture the global optimization objective of the considered multi-agent system, we introduce the following global quadratic performance index
\begin{align} \label{costfun}
	\begin{split}
		J=\int_{0}^{T} \frac{1}{2}[x'(t)Qx(t)+ u'(t)Ru(t)] dt,
	\end{split}
\end{align}
where $x(t)=[x_1'(t),x_2'(t),\cdots,x_N'(t)]'$ is the augmented state, $u(t)=[u_1'(t),u_2'(t),\cdots,u_N'(t)]'$ denotes the augmented control, $Q$ is positive semi-definite matrices, $R$ is positive-definite matrix, all with compatible dimensions.
	In particular, the weighting matrices in this paper can be either diagonal or dependent on the communication topology, i.e., $Q=(c_{ij}Q_{ij})$ and $R=(c_{ij}R_{ij})$, where $c_{ii}=1$, $c_{ij}=1$ if $i$ and $j$ communicate, and $c_{ij}=0$ otherwise. In this case, for each agent $i$, the available information is $\tilde{Q}_i = diag\{0,\ldots,0,NI,0,\ldots,0\}Q$ and $\tilde{R}_i = diag\{0,\ldots,0,NI,0,\ldots,0\}R$,
		which are treated as private information.
	
	Based on the above formulation, we provide a rigorous mathematical description of the privacy requirements under dynamics and local weighting matrices by defining the private information set of agent $i$ as
\begin{align*}
	\mathcal{I}_i(t) = \{x_i(\tau), f_i, g_i, \tilde{Q}_i, \tilde{R}_i \mid \tau \le t \},
\end{align*}
which is accessible only to agent $i$ and can be shared with its communicating neighbors, but unknown to all other non-neighboring agents.
Accordingly, we have the following structure:

	\textbf{Information Structure:}
	For agent $i$, its available information set $\mathcal{S}_i(t)$ is characterized by
	\begin{align*}
		\mathcal{S}_i(t) = \mathcal{I}_i(t) \cup  {\textstyle \bigcup_{j \in \mathcal{N}_i}} \mathcal{I}_j(t).
	\end{align*}

Building on the above preliminaries, the objective of this paper is to address the following problem:
	\begin{problem}
		Under the constraint $u_i(t)\in \mathcal{U}(\mathcal{S}_i(t))$, which means each agent only utilizes information from its own information structure $\mathcal{S}_i(t)$, 
		design a distributed algorithm for each agent $i$ to minimize the global cost function \eqref{costfun} subject to the system \eqref{system_agent}.
	\end{problem}

	\begin{remark}
    The information structure is grounded in practical scenarios and exhibits considerable engineering significance.
     For instance, in sensitive scenarios consisting of multiple unmanned ground vehicles (UGVs) as showned in Fig. \ref{fig:communication}, such as military UGV formations, the real-time position, physical configuration, intended destination, and planned trajectory of each vehicle should be maintained as strictly confidential information.
Unauthorized access or improper information sharing may expose operational intentions and even render the entire system vulnerable to malicious attacks. Therefore, the privacy and security of the private information structure of each agent must be guaranteed.

    \FloatBarrier			
    \begin{figure}[H]
    	\centering		
    	\includegraphics[width=0.47\textwidth, keepaspectratio]{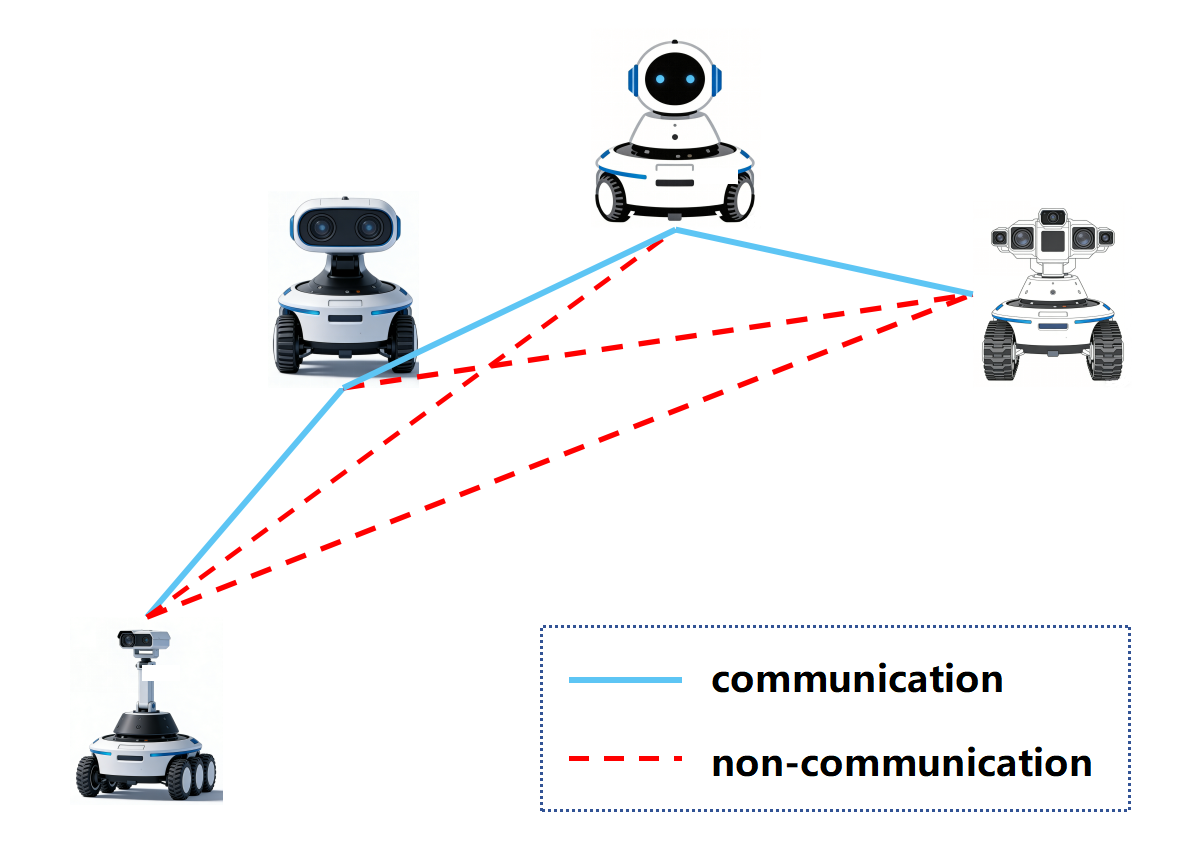} 
    	\caption{Privacy-aware communication in UGV formation. } 
    	\label{fig:communication}
    \end{figure}	
    \FloatBarrier
    		
    \end{remark}

\begin{remark}
	
The global quadratic performance index \eqref{costfun} covers various practical control objectives for multi-agent systems. Three typical cases are discussed below.

(1) Formation consensus: 
The standard consensus-oriented index is $J_1 = \int_{0}^{T} \Sigma_{i,j}(x_i(t)-x_j(t))'Q_{ij}(x_i(t)-x_j(t))dt$, 
where $Q_{ij}$ penalizes state discrepancies between agents to achieve coordination.

(2) Energy efficiency: 
The typical energy-minimization index is $	J_2 =  \int_{0}^{T} u'(t)Ru(t)dt$, 
where $R$ is usually chosen as a diagonal matrix that penalizes the magnitude of each control input to reduce energy consumption.

(3) Collision avoidance: 
The representative safety index is $J_3 =  \int_{0}^{T} \Sigma_{i,j}(d_{ij}(t)-d_0(t))'Q_{ij}(d_{ij}(t)-d_0(t))dt$, 
where $d_{ij}$ denotes the relative distance between the position components extracted from the states $x_i$ and $x_j$, and $Q_{ij}$ penalizes distances smaller than the safe value $d_{0}$.
\end{remark}

\begin{remark}
	It is well known that distributed solving of optimal control problems even in the linear case is NP-hard \cite{NPhard}.
   For the nonlinear optimal control problem, finding an optimal solution is much more difficult.
 To this end, we will propose a novel distributed algorithm to approximate the global optimal controller of nonlinear multi-agent systems.  
\end{remark}

\section{Preliminaries on global optimal controller}	\label{sec3}
	
	In this section, we focus on the design of the global optimal controller that minimizes the cost function \eqref{costfun} subject to the considered nonlinear multi-agent system \eqref{system_agent} under the assumption that all agents possess complete system information. Starting from the objective of minimizing the predefined global cost function, we present the global optimal control policy by solving the HJB equation.
	The theoretical results herein form the foundation for the subsequent design of distributed algorithm under the partial information structure. 
	
	We first present the system dynamics under the full-information assumption, and then derive the corresponding control input. Specifically, the dynamics of the augmented state $x(t)$ is described as:
	\begin{align}
		\dot{x}(t)=f(x(t))+g(x(t))u(t), \ \  x(0)=x_0,\label{system_0}
	\end{align}
	with 	\begin{align*}
		f(x(t))=&\frac{1}{N}\sum_{i=1}^{N} \mathfrak{f}_i(x_i(t)),\\
		g(x(t))=&\frac{1}{N}\sum_{i=1}^{N} \mathfrak{g}_i(x_i(t)),\\
		x_0=&\frac{1}{N}\sum_{i=1}^{N}\mathfrak{x}_{i0},  
	\end{align*}
	where 
	\begin{align*}
		&\mathfrak{f}_i(x_i(t))=\begin{pmatrix}
			0 \\
			\vdots  \\
			Nf_i(x_i(t))\\
			\vdots\\
			0
		\end{pmatrix},   	
		\mathfrak{x}_{i0}=\begin{pmatrix}
			0 \\
			\vdots  \\
			Nx_{i0}\\
			\vdots\\
			0
		\end{pmatrix},\\
		&\mathfrak{g}_i(x_i(t))=\begin{pmatrix}
			0 &  &  &  & \\
			&  \ddots &  &  & \\
			&  & Ng_i(x_i(t))&  & \\
			&  &  & \ddots  & \\
			&  &  &  &0
		\end{pmatrix}.
	\end{align*}

	At this point, for the given system \eqref{system_0}, the core of solving the nonlinear optimal control problem lies in deriving and solving the corresponding HJB equation:
	\begin{align} \label{HJB}
		\frac{\partial V(t,x)}{\partial t} + \min_{u}H(x,\nabla_xV,u(t,x))=0,
	\end{align}
	with the terminal condition $V(T,x) = 0$, where the Hamiltonian $H$ is given by:
	\begin{align*} 
		\begin{split}
			H(x,\nabla_xV,u(t,x))=&[\nabla_xV(t,x)]' F(t,x)+l(t,x),
		\end{split}
	\end{align*}
	while 
	\begin{align*}
		F(t,x)& = f(x)+g(x)u(t,x),\\
		l(t,x)&= \frac{1}{2}x'(t)Qx(t)+\frac{1}{2}u'(t,x)Ru(t,x).
	\end{align*}
		
	For the analysis that follows, we assume throughout this paper that the HJB equation \eqref{HJB} admits a continuously differentiable solution, as noted in \cite{C1}.

	\begin{lemma} \label{lemma1}
		The globally optimal controller that minimizes the cost function \eqref{costfun} subject to \eqref{system_0} with $x(0)=x_0$ is given by:
		\begin{align}\label{G_control}
			u(t,x)=-R^{-1}g'(x) \nabla_xV(t,x).
		\end{align}	
	\end{lemma}
	\textit{\textbf{Proof.}} The proof is referred to \cite{VI} and omitted herein.	
	\hfill $\blacksquare$

	From Lemma 1, the design of optimal controller \eqref{G_control} relies on the solution $V(t,x)$ to the HJB equation \eqref{HJB}, i.e.,
		\begin{align}
			\frac{\partial V(t,x)}{\partial t} +[\nabla_xV(t,x)]' [ f(x)-g(x)R^{-1}g'(x)	\nonumber \\ \times \nabla_xV(t,x)   ]
			+ [\frac{1}{2}x'(t)Qx(t)+\frac{1}{2} (\nabla_xV(t,x))'	\nonumber \\ \times g(x)  R^{-1}g'(x) \nabla_xV(t,x) ] =0.	
	\end{align}
	
	In the following lemma, we present a solution scheme adopted from the method developed in \cite{VI} to address this partial differential equation (PDE), which will be useful in the design of the distributed algorithm.

	\begin{lemma}[Value Iteration]\label{lemma2}
		Under the given initialization $V^{0}(t,x) = 0$ and the terminal condition $V^{k+1}(T,x) = 0$, the value iteration
		\begin{align} \label{linear PDE}
				\frac{\partial V^{k+1}(t,x)}{\partial t} + [\nabla_xV^{k+1}(t,x)]' F^{k}(t,x)\nonumber  \\
				+l^{k}(t,x)=0,
		\end{align}
		with
		\begin{align}
			u^{k}(t,x)&= -R^{-1}g'(x)\nabla_xV^{k}(t,x), \label{uk}\\
			F^{k}(t,x)&=f(x)+g(x)u^{k}(t,x), \label{Fk}\\
			l^{k}(t,x)&=\frac{1}{2}x'(t)Qx(t)+\frac{1}{2}(u^k(t,x))'Ru^k(t,x), \label{lk}
		\end{align}
		yields a sequence $\{V^{k+1}(t,x)\} $ that satisfies
		\begin{align}
			\lim_{k \to \infty} V^{k+1}(t,x)=V(t,x).	
		\end{align}
		\end{lemma}
	\textit{\textbf{Proof.}}
The proof follows from Theorem 1 in \cite{VI} and is thus omitted. \hfill $\blacksquare$

	\section{Main result}\label{sec4}
	This section develops a distributed optimal control algorithm for the nonlinear multi-agent systems under study. 
	Specifically, under the constrained information structure $\mathcal{S}_i(t)$, we derive a distributed approximation of the global optimal control by numerically solving the HJB equation \eqref{Vkis} developed below, which enables each agent to iteratively  approximate the global optimal controller.
	To facilitate this iterative approximation, we present the following equations:
	\begin{align}  \label{Vkis}
			\frac{\partial V^{k+1}_{i,s}(t,x^{k}_{i,s})}{\partial t} + [\frac{\partial V^{k+1}_{i,s}(t,x^{k}_{i,s})}{\partial x^{k}_{i,s}} ]'F^{k}_{i,s}(t,x^{k}_{i,s}) \nonumber  
			\\	+l^{k}_{i,s}(t,x^{k}_{i,s})=0,
	\end{align}
	with 
%%%%%%%%%%%%%%%%%%%%%%%%%%%%%%%%%%%%%%%%%%%
		\begin{align}
		 \label{xkis}
			x^{k}_{i,s}(t) =& x^{k}_{i,s-1}(t) +  \delta_s [\mathfrak{x}_{i0} + \int_{0}^{t}   \mathfrak{f}_i(x^{k}_{i,s-1}(\tau))\nonumber  \\
			&+ \mathfrak{g}_i(x^{k}_{i,s-1}(\tau))  u^{k}_{i,s-1}(\tau,x^{k}_{i,s-1})  d\tau \nonumber  \\
			&- x^{k}_{i,s-1}(t) ]  +\frac{1}{ \kappa} \sum_{j\in\mathcal{N}_i  } [x^{k}_{j,s-1}(t) \nonumber  \\&-x^{k}_{i,s-1}(t)], 
	\\
	       \label{Fkis}
			F^{k}_{i,s}(t,x^{k}_{i,s}) =& F^{k}_{i,s-1}(t,x^{k}_{i,s-1}) + \delta_s [ \mathfrak{f}_i(x^{k}_{i,s}(t))\nonumber  \\
			& + \mathfrak{g}_i(x^{k}_{i,s}(t)) u^{k}_{i,s}(t,x^{k}_{i,s}) \nonumber 
			\\
			&- F^{k}_{i,s-1}(t,x^{k}_{i,s-1}) ]\nonumber  \\
			& + \frac{1}{ \kappa} \sum_{j\in\mathcal{N}_i  } [F^{k}_{j,s-1} (t,x^{k}_{j,s-1}) \nonumber  \\
			&-F^{k}_{i,s-1}(t,x^{k}_{i,s-1})   ]	,
		\\
		%		\end{align}
	%		\begin{align}
		%
		 \label{lkis}
		l^{k}_{i,s}(t,x^{k}_{i,s})
			 =& l^{k}_{i,s-1}(t,x^{k}_{i,s-1})  +  \delta_s [ \frac{1}{2}(x^{k}_{i,s}(t))'\tilde{Q}_i \nonumber  \\
			 &\times x^{k}_{i,s}(t) + \frac{1}{2}(u^{k}_{i,s}(t,x^{k}_{i,s}))'\tilde{R}_i\nonumber  \\
			 &\times  u^{k}_{i,s}(t,x^{k}_{i,s}) -l^{k}_{i,s-1}( t,x^{k}_{i,s-1})]\nonumber  \\
			&+ \frac{1}{ \kappa} \sum_{j\in\mathcal{N}_i  } [l^{k}_{j,s-1}(t,x^{k}_{j,s-1}) \nonumber \\
			& -l^{k}_{i,s-1}(t,x^{k}_{i,s-1})],	
		\\
		\label{ukis}
			u^{k}_{i,s}(t,x^{k}_{i,s}) = & -\bar{R}_ig'_i(x^{k}_{i,s})\nabla_xV^{k}_{i,s}(t,x^{k}_{i,s}),
	\end{align}
	where
	\begin{align*}
		&\mathfrak{x}_{i0} = [0',\ldots,0',Nx'_{i0},0',\ldots,0']', \\
%		&\tilde{Q}_i = diag\{0,\ldots,0,NI,0,\ldots,0\}Q,\\
%		&\tilde{R}_i = diag\{0,\ldots,0,NI,0,\ldots,0\}R,\\
		&\bar{R}_i  = diag\{0,\dots,0,R_i^{-1},0,\dots,0\},
	\end{align*}
	while $k = 0,1,2,...$, $s=1,2,3,...$, $V^{0}_{i,s}(\cdot)=0 $,  $x^{k}_{i,0}(\cdot)=0$, $F^{k}_{i,0}(\cdot)=0$, $l^{k}_{i,0}(\cdot)=0$, 
	$\delta_s$ is a step size satisfying $\sum_{s=0}^{\infty} \delta_s =\infty$, $\sum_{s=0}^{\infty} \delta_s^2<\infty$ and $\kappa>0 $ is appropriately chosen to guarantee that the matrix $I_N - \frac{1}{\kappa}\mathcal{L} - \frac{1}{N}1_N1_N'$ is Hurwitz, thereby ensuring the convergence of the distributed iteration process, which is established in the following result.

	\begin{theorem} \label{Theorem1}
%		Given any initial time $t\in [0, T]$ and initial state $\mathfrak{x}_{i0}$,
	For any  $\varepsilon>0$, there exists a positive integer $W$
		such that for $k>W$, $s>W$, 
		\begin{align} 
			\left \| x^{k}_{i,s}(t) - x(t) \right \|< \varepsilon, \label{conx} \\
			\left \| F^{k}_{i,s}(t, x^{k}_{i,s}) - F(t,x) \right \|  < \varepsilon,\label{conF} \\
			\left \| l^{k}_{i,s}(t,x^{k}_{i,s}) - l(t,x) \right \|  < \varepsilon,\label{conl} \\
			\left \| V^{k+1}_{i,s}(t,x^{k}_{i,s}) - V(t,x) \right \|  < \varepsilon,\label{conV}
		\end{align}
		where $x(t)$, $F(t,x)$, $l(t,x)$ and $V(t,x)$ are the corresponding solutions under the global optimal controller \eqref{G_control}.
		\end{theorem}	
		\textit{\textbf{Proof.}}
		The detailed proof is given in Appendix A. \hfill $\blacksquare$

	For notational simplicity, we introduce the definitions:
	\begin{align*}
		x^{\infty}_i(t) & \doteq  x^{W}_{i,W}(t),\\
		F^{\infty}_i(t)	&  \doteq  F^{W}_{i,W}(t,x^{W}_{i,W}),\\
	    l^{\infty}_i(t)&  \doteq  l^{W}_{i,W}(t,x^{W}_{i,W}),\\	
		V^{\infty}_i(t)&  \doteq  V^{W}_{i,W}(t,x^{W}_{i,W}).
	\end{align*}
	Based on these notations, the distributed optimal control design for agent $i$ is presented below.
	\begin{theorem}\label{themU}
		For any $\varepsilon>0$, there exists a positive integer $W$ such that the distributed optimal controller for each agent $i$ given by
		\begin{align}
	\bar{u}_i(t) \hspace{-3pt}=\hspace{-3pt} -\hspace{-1pt}[0,\ldots, 0,R_i^{-1}g'_i(\check{x}_i(t)),0,\ldots,0]\hspace{-1pt}\nabla_xV_i^\infty(t),
		\end{align}
		satisfies
		\begin{align} \label{conu}
			\left\|U(t) - u(t,x)\right\| < \varepsilon,
		\end{align}
		where $\check{x}_i(t)= [0,\ldots,0,I_{n},0,\ldots,0]x_i^\infty(t)$ and $U(t) = [\bar{u}'_1(t),\bar{u}'_2(t),\ldots,\bar{u}'_N(t)]'$.
\end{theorem}
	\textit{\textbf{Proof.}} Using the definitions of $U(t)$ and $u(t,x)$ in Lemma \ref{lemma1}, we have the following expansion:
		\begin{align*}	
			U&(t) - u(t,x)\\
%				=&-\begin{pmatrix}
%					[R_1^{-1}g'_{1}( \check{x} _{1}(t)),0,\ldots ,0]\nabla_xV_1^\infty(t) \\
%					\vdots \\
%					[0,\ldots,0,R_i^{-1}g'_{i}( \check{x} _{i}(t)),0,\ldots ,0]\nabla_xV_i^\infty(t)\\
%					\vdots \\
%					[0,\ldots ,0,R_N^{-1}g'_{N}( \check{x} _{N}(t))]\nabla_xV_N^\infty(t)\\
%				\end{pmatrix}\\
%				&+R^{-1}g'(x) \nabla_xV(t,x)
%				%
%			\end{align*}
%		\begin{align*}
%			=& 
%				-\begin{pmatrix}
%					R_1^{-1}g'_{1}( \check{x}_{1}(t)) \frac{\partial V_1^\infty(t) }{\partial x_1} \\
%					\vdots \\
%					R_i^{-1}g'_{i}( \check{x}_{i}(t))\frac{\partial V_i^\infty(t) }{\partial x_i}\\
%					\vdots \\
%					R_N^{-1}g'_{N}( \check{x}_{N}(t))\frac{\partial V_N^\infty(t) }{\partial x_N}\\
%				\end{pmatrix}
%				+\begin{pmatrix}
%					R_1^{-1}g'_{1}( x_{1}(t)) \frac{\partial V(t,x) }{\partial x_1} \\
%					\vdots \\
%					R_i^{-1}g'_{i}( x_{i}(t))\frac{\partial V(t,x) }{\partial x_i}\\
%					\vdots \\
%					R_N^{-1}g'_{N}( x_{N}(t))\frac{\partial V(t,x) }{\partial x_N}\\
%				\end{pmatrix} \\
				%
				=&
				\begin{pmatrix}
				 -	R_1^{-1} [g'_{1}( \check{x}_{1}(t)) \frac{\partial V_1^\infty(t) }{\partial x_1}- g'_{1}( x_{1}(t)) \frac{\partial V(t,x) }{\partial x_1} ]	 \\
					\vdots \\
				 -	R_i^{-1} [g'_{i}( \check{x}_{i}(t)) \frac{\partial V_i^\infty(t) }{\partial x_i}- g'_{i}( x_{1}(t)) \frac{\partial V(t,x) }{\partial x_i} ]	\\
					\vdots \\
				 -	R_N^{-1} [g'_{N}( \check{x}_{N}(t)) \frac{\partial V_N^\infty(t) }{\partial x_N}- g'_{N}( x_{N}(t)) \frac{\partial V(t,x) }{\partial x_N} ]	\\
				\end{pmatrix}.
		\end{align*}
%%%%%%%%%%%%%%%%%%%%%%	
 Then, by applying the Cauchy-Schwarz inequality and defining  $C_1 = \max_i \|R_i^{-1}\|^2$, we obtain
	\begin{align} \label{U-u}
	&	\left\|U(t) - u(t,x)\right\|^2 \nonumber \\
	=&  \sum_{i=1}^{N}  \left\|	-	R_i^{-1} [g'_{i}( \check{x}_{i}(t)) \frac{\partial V_i^\infty(t) }{\partial x_i}- g'_{i}( x_{1}(t)) \frac{\partial V(t,x) }{\partial x_i} ]  \right\|^2 
	\nonumber \\
		\le& C_1  \sum_{i=1}^{N}  \left\| g'_{i}( \check{x}_{i}(t)) \frac{\partial V_i^\infty(t) }{\partial x_i}- g'_{i}( x_{i}(t)) \frac{\partial V(t,x) }{\partial x_i}   \right\| ^2
\nonumber	\\
	 \le& 2 C_1  \sum_{i=1}^{N}  \{	
			\left\|	g'_{i}( \check{x}_{i}(t)) \right\|^2
			\left\| \frac{\partial V_i^\infty(t) }{\partial x_i}- \frac{\partial V(t,x) }{\partial x_i}  \right\|^2 \nonumber \\
	&	+\left\|g'_{i}( \check{x}_{i}(t)) - g'_{i}( x_{1}(t))\right\|^2  
	      \left\| \frac{\partial V(t,x) }{\partial x_i} \right\|^2 \}   .
	\end{align}
%%%%%%%%%%%%%%%%%%%%%%		
	Since the system operates over the finite time horizon $[0,T]$ with the fixed initial state, the state trajectory $\{x(t)\}$, $\{x_{i,s}^{k}(t)\}$ are bounded. Together with the continuous differentiability of $V(\cdot)$, $V_{i,s}^{k}(\cdot)$, $g_i(\cdot)$ and the convergence results in Theorem 1, there exist positive constants $C$ such that for all $i$,
	\begin{align*}
		& \begin{Vmatrix}  g'_{i}( \check{x}_{i}(t))\end{Vmatrix}\le C, \ \ \ 
		\begin{Vmatrix}\frac{\partial V_i^\infty(t) }{\partial x_i}- \frac{\partial V(t,x) }{\partial x_i} \end{Vmatrix}\le \frac{1}{2C\sqrt{NC_1}} \varepsilon ,\\
	& 	\begin{Vmatrix}\frac{\partial V(t,x) }{\partial x_i}\end{Vmatrix}\le C ,\ \ \  \ \
		\begin{Vmatrix}	g'_{i}( \check{x}_{i}(t)) - g'_{i}( x_{i}(t))\end{Vmatrix}\le \frac{1}{2C\sqrt{NC_1}} \varepsilon.
	\end{align*}
%%%%%%%%%%%%%%%%%

 Substituting these bounds into the above inequality \eqref{U-u}, we have
	$	\left\|U(t) - u(t,x)\right\|^2 \le \varepsilon^2 $, which implies that \eqref{conu} holds.
 The proof is now completed. \hfill 	 $\blacksquare$

	In light of the above analysis, the distributed algorithm for solving the optimal controller of each agent $i$ is shown in Algorithm \ref{alg:dis}.	
	
%%%%%%%%%%%%%%%%%%%%%%%%%%%%%%%%%%%%%%%		
	\begin{algorithm}[]
		\caption{Distributed Optimal Controller}
		\label{alg:dis}
		\begin{algorithmic}[1]
			\item \textbf{Input:} 
		Sufficiently large positive integer $W$, initial value $x^{k}_{i,0}(\cdot)=0$, $F^{k}_{i,0}(\cdot)=0$, $l^{k}_{i,0}(\cdot)=0$, $V^{0}_{i,s}(\cdot)=0 $.
			\item \textbf{Output:} The distributed optimal controller $\bar{u}_i(t)$.

		     \item \textbf{for} {$k = 0,1,2,...,W  $} \textbf{do}
			\item \hspace{0.2cm} \textbf{for} {$s = 0,1,2,...,W$} \textbf{do}
			
			\item \hspace{0.4cm} Compute $x^{k}_{i,s}(t,x^{k}_{i,s})$, $F^{k}_{i,s}(t,x^{k}_{i,s})$, $l^{k}_{i,s}(t,x^{k}_{i,s})$, $u^{k}_{i,s}(t,x^{k}_{i,s})$ with  \eqref{xkis}-\eqref{ukis}.	
			\item \hspace{0.55cm} Solve $V^{k+1}_{i,s}(t,x^{k}_{i,s})$ by  HJB equation \eqref{Vkis}.
			\item \hspace{0.2cm} \textbf{end for}
			\item \textbf{end for}	
		\end{algorithmic}   
	\end{algorithm}

	\begin{remark}
		Running the Algorithm \ref{alg:dis}, we can obtain the distributed optimal controller $\bar{u}_i(t)$ that satisfies \eqref{conu}. The detailed procedure for solving the HJB equation \eqref{Vkis} involved in the algorithm will be presented in Section~\ref{sec5}.
	\end{remark}

	\section{Application simulation of multi-UGV system}\label{sec5}
	
	This section presents a practical numerical simulation on the multi-agent system composed of five UGVs to verify the effectiveness of 
	the proposed distributed algorithm.
%	 \ref{alg:dis}.

Specifically, the five-UGV system adopted in the simulation is a typical distributed cooperative control system, whose dynamics model for the $i^{th}$ UGV is given in \cite{UGV,UGVs} as follows: 
\begin{align*}
	&	\dot{r}_{xi}(t) = v_i(t)cos(\theta_i(t)),
	\\
	&	\dot{r}_{yi}(t) = v_i(t)sin(\theta_i(t)),\\
	&	\dot{\theta}_{i}(t) = \omega_i(t),
\end{align*}
where $(r_{xi}(t), r_{yi}(t))$, $\theta_{i}(t)$, $(v_i(t),\omega_i(t))$ denote the Cartesian position, orientation, and the linear and angular velocity of the $i^{th}$ UGV, respectively. Obviously, this model is consistent with system \eqref{system_agent}, where the state $x_i(t)=[r'_{xi}(t), r'_{yi}(t), \theta'_{i}(t)]'$, control input $u_i(t)=[v'_{i}(t), \omega'_i(t)]'$, the system functions $f_i(x_i(t))=0$ and $g_i(x_i(t))=\begin{bmatrix}
	cos(\theta_i(t))	& 0\\
	sin(\theta_i(t))	& 0\\
	0	&1
\end{bmatrix}$.	
The optimal control goal is to minimize both the relative states and energy consumption for the five-UGV system with given initial states
$x_{10}=[10,2,\frac{\pi}{8}]',
x_{20}=[14,4,\frac{\pi}{4}]',
x_{30}=[15,2,\frac{\pi}{3}]',
x_{40}=[12,1,\frac{\pi}{4}]',
x_{50}=[10,0,\frac{\pi}{8}]'$, where the communication topology of the five UGVs is shown in Fig. \ref{fig:topo},
and the weighting matrices of the cost function \eqref{costfun} are set as
\begin{align*}
		Q=\begin{pmatrix}
		&2I,&-2I,&0,&0,&0 \\
		&-2I,&6I,&-2I,&0,&-2I \\
		&0,&-2I,&2I,&0,&0 \\
		&0,&0,&0,&2I,&-2I \\
		&0,&-2I,&0,&-2I,&4I \\
	\end{pmatrix}, 
	R=0.01I.
\end{align*}
\FloatBarrier			
\begin{figure}[H]
	\centering		
	\includegraphics[width=0.49\textwidth, keepaspectratio]{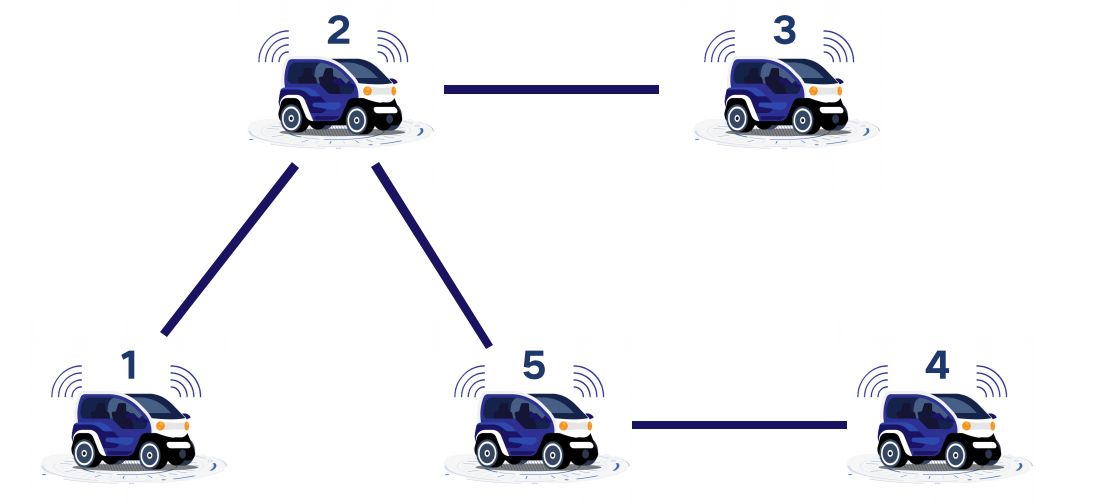} 
	\caption{Communication topology among UGVs. } 
	\label{fig:topo}
\end{figure}	
\FloatBarrier

Building upon the specified simulation parameters, the numerical experiment for distributed optimal control of the five-UGVs system is implemented based on the framework of Algorithm \ref{alg:dis}, with a key step involving the solution to the HJB equation \eqref{Vkis}. 
As this equation is a PDE for which an analytical solution is generally not obtainable, we adopt a neural network-based (NN) numerical method to approximate the solution.
In particular, the solution to the PDE \eqref{Vkis} is approximated by the neural network $\hat{V}^{k+1}_{i,s}(t,x) = \Theta^{k+1}_{i,s}(t)\Psi(x)$,
%\begin{align*}
%	\hat{V}^{k+1}_{i,s}(t,x) = \Theta^{k+1}_{i,s}(t)\Psi(x),
%\end{align*}
where 
$\Theta^{k+1}_{i,s}(t) =[\theta^{k+1}_{i,s,1}(t),\cdots , \theta_{i,s,j}^{k+1}(t),\cdots,\theta_{i,s,M}^{k+1}(t)]$, $\Psi(x)=[ 
\varphi_1'(x), \cdots,\varphi_j'(x) ,\cdots,  \varphi_{M}'(x)]' $,
the inverse multiquadric radial basis function (IMQ-RBF) $\varphi_j(x)$ is described as 
\begin{align} \label{RBF}
	\varphi_j(x)= \frac{1}{\sqrt{\left \| x-c_j \right \|_2^2 +z^2} },
\end{align}
with the $j^{th}$ basis center $c_j$ and the shape parameter $z>0$.
The detailed implementation steps are presented in Algorithm \ref{alg:IMQ} below.

\begin{algorithm}[H]
	\caption{ NN Approximation for $V^{k+1}_{i,s}(\cdot)$}
	\label{alg:IMQ}
	\begin{algorithmic}[1]
		\item \textbf{Input:}
		Collocation and RBF centers $\{x_j\}_{j=1}^M$, time step $t_n \in \{ n\Delta t |n =  0,1,\dots,N_t \}$, $ \nabla V_{i,s}^{0}(t_n,x_j)=0 $, $u^{0}_{i,s}(t_n,x_j)=0$, $\Theta^{k}_{i,s}(T) = 0$.
		\item \textbf{Output:} Value function ${V}^{k+1}_{i,s}(t,x)\approx  \Theta^{k+1}_{i,s}(t)\Psi(x)$.
%		 and control $u^k_{i,s}(t,x)$.					
		\item \textbf{for }{$ k = 0,1,2,3,...$}   \textbf{do}
        \item \hspace{0.005cm}  \textbf{for }{$ n = N_t-1, \dots,0$}     \textbf{do}
			 \item \hspace{0.15cm} Compute cost vector $L^{k}(t_{n})$:
			$	L_j^{k}(t_{n}) = l^k_{i,s}(t_{n},x_j)$ .
			 \item  \hspace{0.15cm}  Compute $ \nabla\varphi_b(x_j)$ and obtain matrix $B^{k}(t_{n})$ with $B_{bj}^{k}(\hspace{-0.03cm}t_{n}\hspace{-0.03cm})\hspace{-0.1cm}=\hspace{-0.1cm} \nabla\varphi_b(x_j) F^k_{i,s}(t_{n},x_j)$.
         \item \hspace{0.15cm} Construct collocation matrix $A$ with
$ A_{bj} = \varphi_b(x_j)$ by \eqref{RBF}.
		\item \hspace{0.15cm}  Solve for coefficients $\Theta^{k+1}_{i,s}(t_n) $ via
			$	\Theta^{k+1}_{i,s}(t_n) = [ A - \Delta t B^{k}(t_{n})]^{-1} [\Theta^{k+1}_{i,s}(t_{n+1}) A + \Delta t L^{k}(t_{n}) ]$.			
			\item \hspace{0.15cm}  Compute
						$ \nabla V^{k+1}_{i,s}(t_n,x_j)$
			and $u^{k}_{i,s}(t_n,x_j)$.			
				\item \hspace{0.005cm}  \textbf{end for}
			\item \textbf{end for}
	\end{algorithmic}
\end{algorithm}
%Construct collocation matrix

By setting parameters as $\kappa =2.5$, step size $\delta_s =\frac{1}{s}$,  shape parameter $z=70$, finite time horizon $T=2$, discrete time step $N_t=101$ and $W=2000$, Algorithm \ref{alg:dis} and Algorithm \ref{alg:IMQ} are implemented to generate the norm trajectory curves of the matrix terms $x^{\infty}_i(t)$, $F^{\infty}_i(t)$, $l^{\infty}_i(t)$, $V^{\infty}_i(t)$, which are illustrated in Fig. \ref{fig:x} to Fig. \ref{fig:V}, respectively. 
The norm trajectory curve of the distributed controller $U(t)$ designed in Theorem \ref{themU} is displayed in Fig. \ref{fig:u}.

     It can be clearly observed that $x^{\infty}_i(t)$, $F^{\infty}_i(t)$, $l^{\infty}_i(t)$ and $V^{\infty}_i(t)$ derived from the proposed distributed algorithms converge to their global centralized counterparts $x(t)$, $F(t,x)$, $l(t,x)$ and $V(t,x)$, respectively. The optimal controller $U(t)$ designed on these bases converges to the global centralized optimal controller $u(t,x)$. 
	
Furthermore, a comparative simulation is conducted with the consensus protocol in \cite{protocol2}.
The performance index of the proposed approach is $110.25$, while that of the protocol in \cite{protocol2} is $174.04$.
		It is evident that the proposed method yields a lower performance index, which further verifies the superior control performance of the proposed distributed algorithm over the traditional method.

\begin{figure}[H]
	\begin{center}
		\includegraphics[width=0.47\textwidth, keepaspectratio]{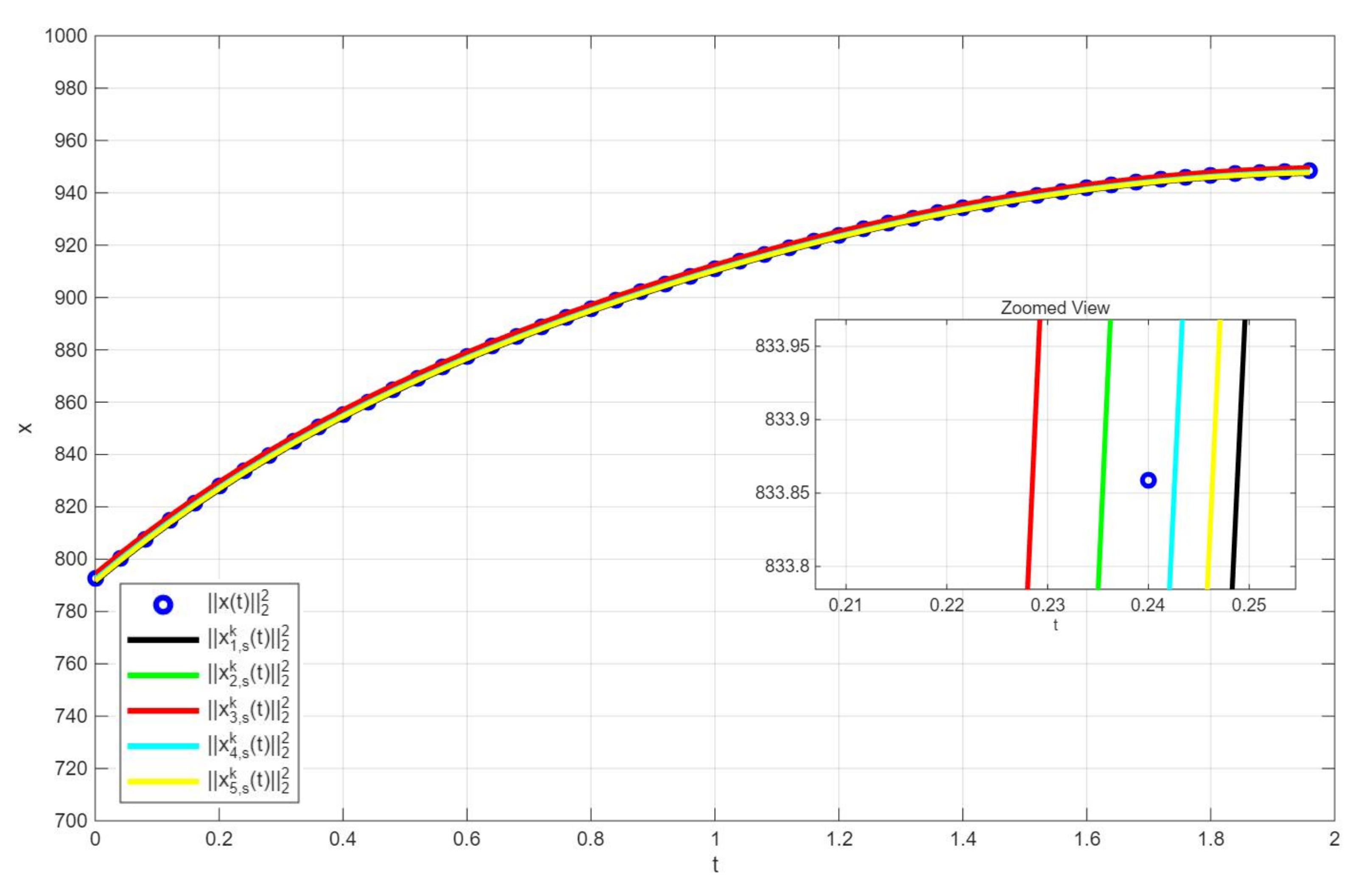}
		\caption{The trajectories of $\left \| x^k_{i,s} (\cdot)\right \| $.}
		\label{fig:x}
	\end{center}
\end{figure}

\begin{figure}[H]
	\begin{center}
		\includegraphics[width=0.47\textwidth, keepaspectratio]{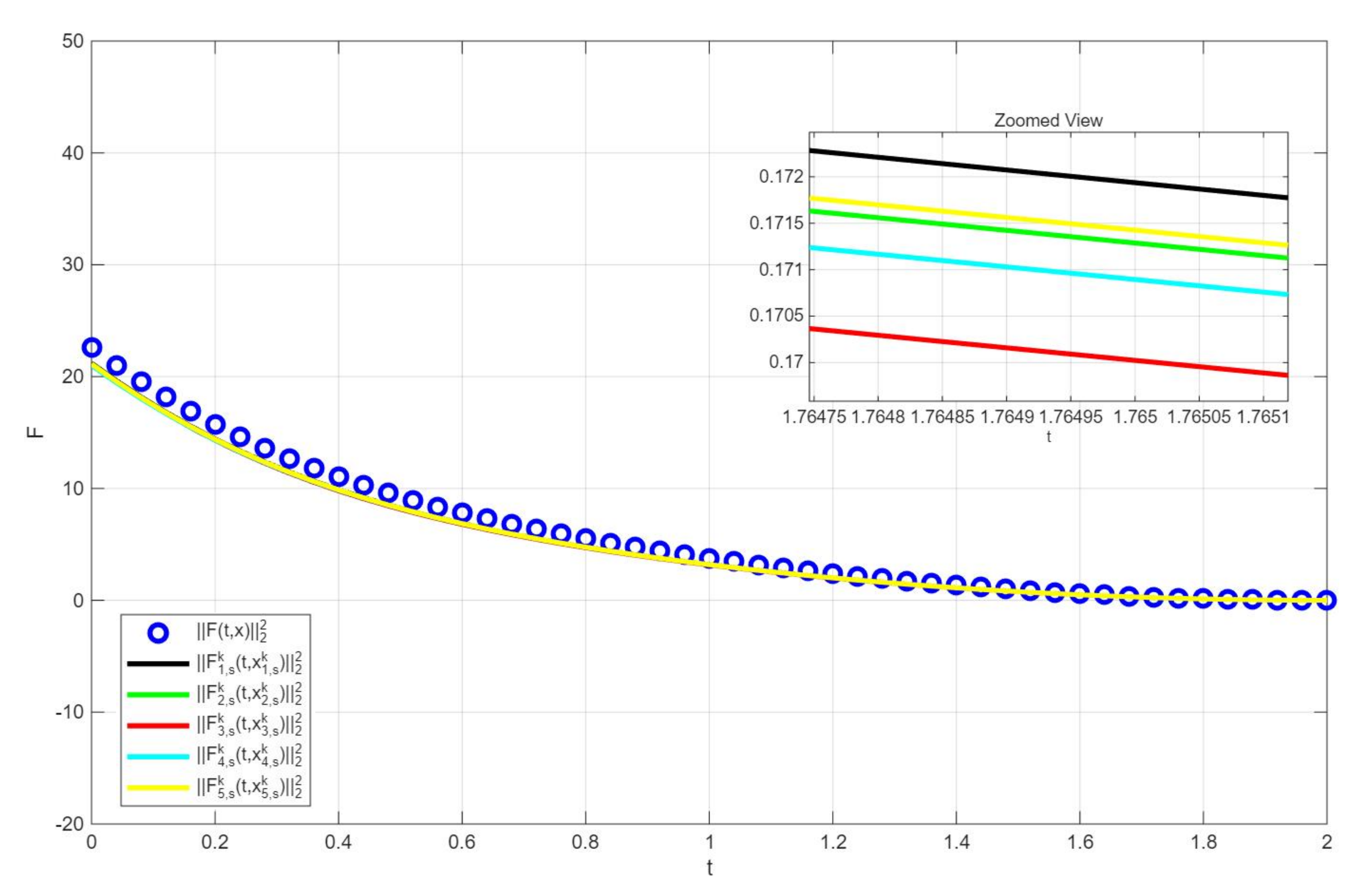}
		\caption{The trajectories of $\left \| F^k_{i,s} (\cdot)\right \| $.}
		\label{fig:F}
	\end{center}
\end{figure}

\begin{figure}[H]
	\begin{center}
		\includegraphics[width=0.47\textwidth, keepaspectratio]{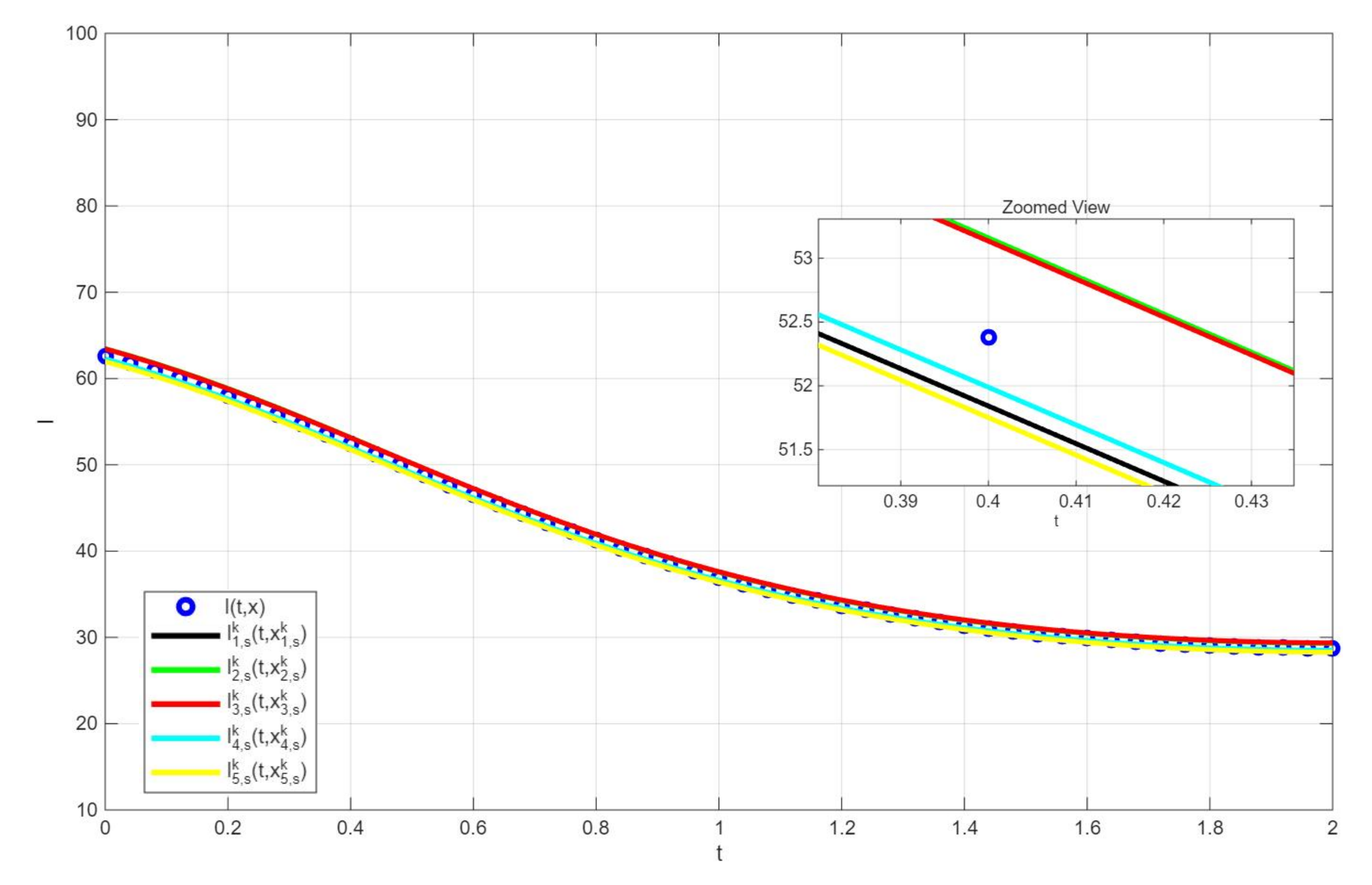}
		\caption{The trajectories of $\left \| l^k_{i,s} (\cdot)\right \| $.}
		\label{fig:l}
	\end{center}
\end{figure}

\begin{figure}[H]
	\begin{center}
		\includegraphics[width=0.47\textwidth, keepaspectratio]{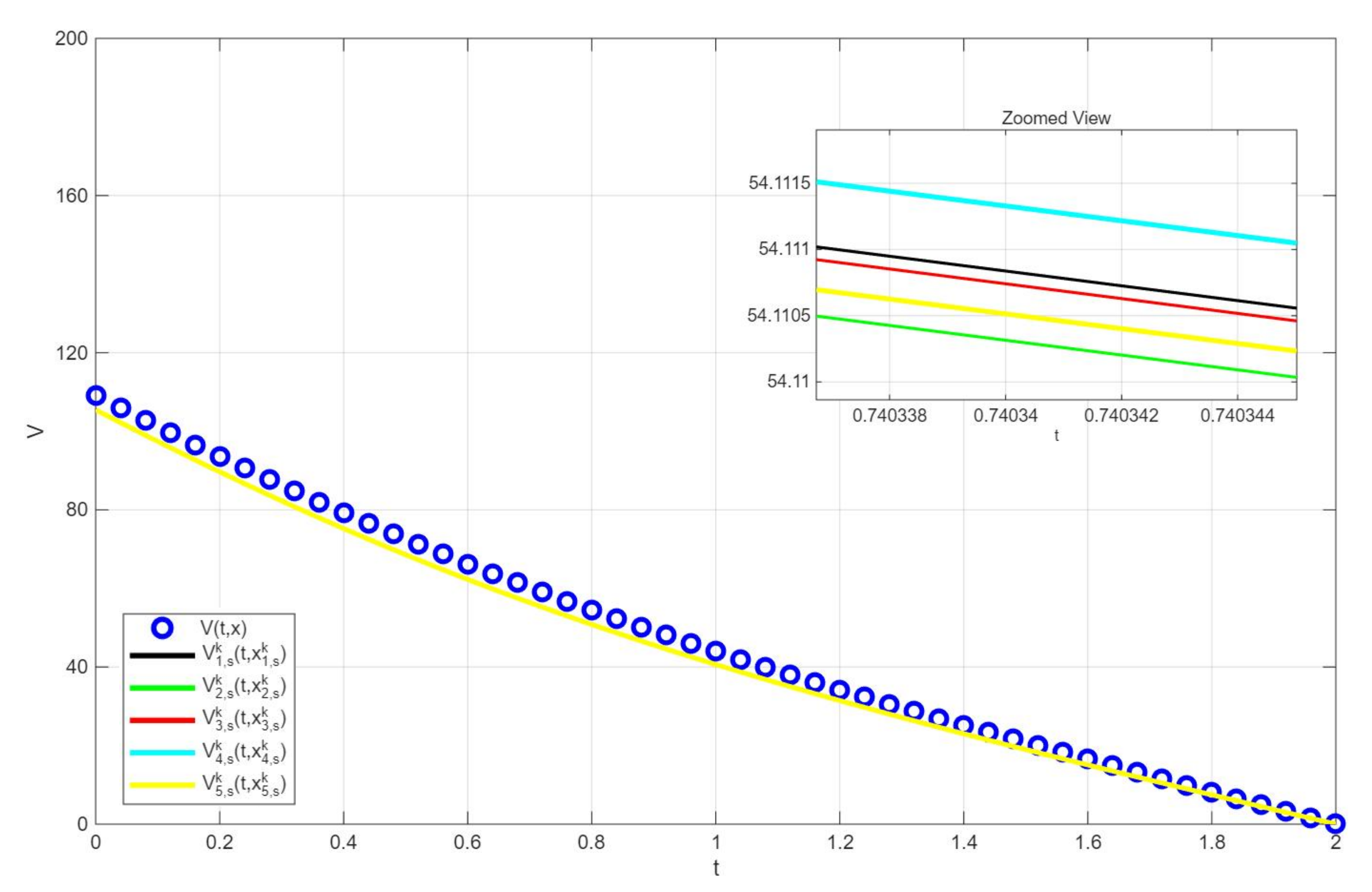}
		\caption{The trajectories of $\left \| V^k_{i,s} (\cdot)\right \| $.}
		\label{fig:V}
	\end{center}
\end{figure}

\begin{figure}[H]
	\begin{center}
		\includegraphics[width=0.47\textwidth, keepaspectratio]{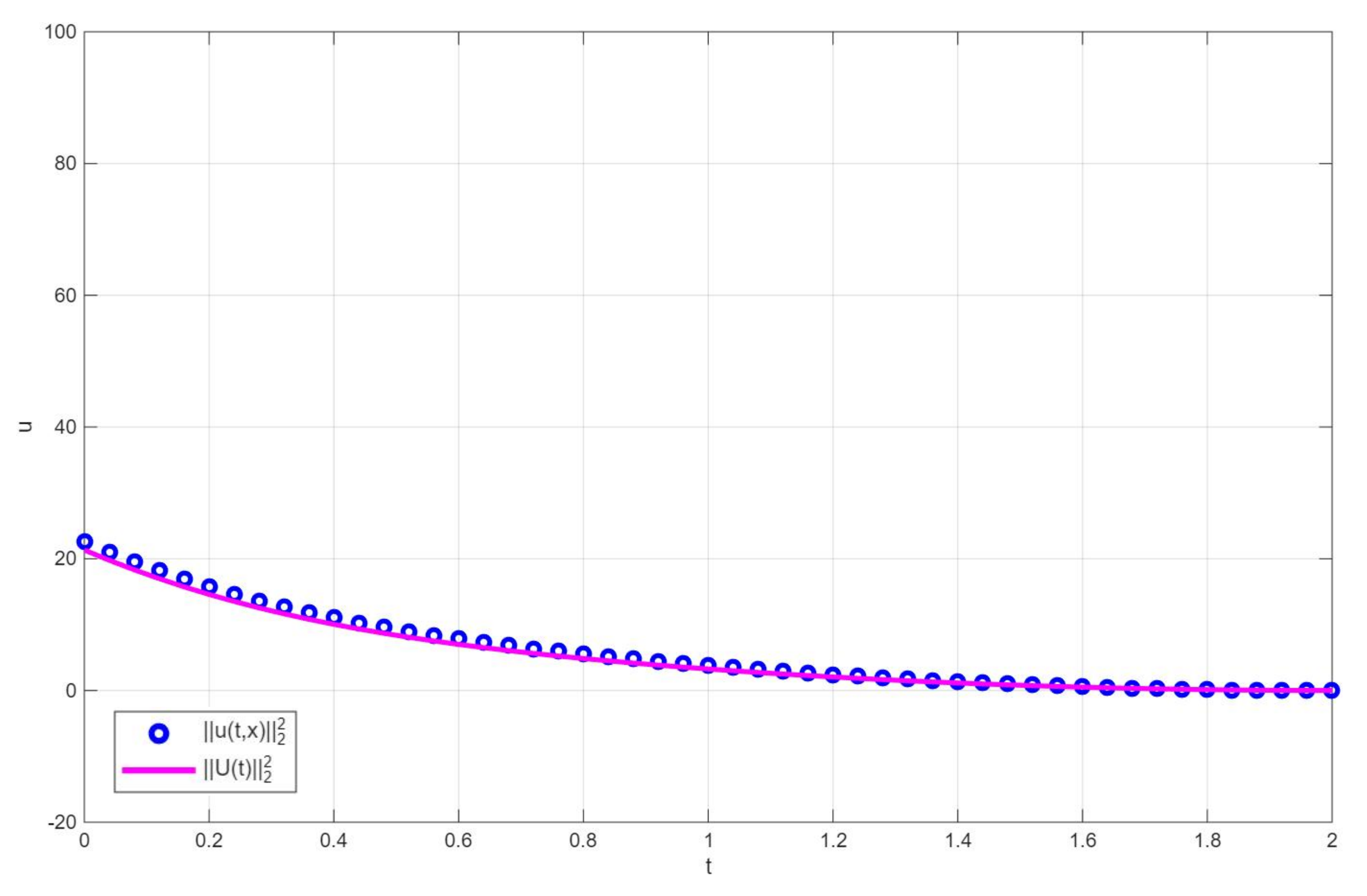}
		\caption{The trajectories of $\left \| u^k_{i,s} (\cdot)\right \| $.}
		\label{fig:u}
	\end{center}
\end{figure}

%\begin{figure}
%	\begin{center}
%		\includegraphics[width=0.49\textwidth, keepaspectratio]{UGV5_PER_INDIX_1.png}
%		\caption{Performance index comparison between the proposed algorithm and \cite{protocol2}.}
%		\label{fig:proto}
%	\end{center}
%\end{figure}
%

%\bibliographystyle{plain}        % Include this if you use bibtex 
%\bibliography{autosam}           % and a bib file to produce the 
%                                 % bibliography (preferred). The
%                                 % correct style is generated by
%                                 % Elsevier at the time of printing.

\appendix
\section{Proof of Theorem \ref{Theorem1} }    % Each appendix must have a short title.
For the convergence analysis, we first fix the number of iterations $k$ in the value iteration. It is then aimed to show that for an arbitrary positive scalar $\varepsilon>0$, there exists a positive integer $W$ such that 
\begin{align}
	\left \| x^{k}_{i,s}(t) - x^{k}(t) \right \| < \varepsilon, \label{kx}\\
	\left \| F^{k}_{i,s}(t, x^{k}_{i,s}) - F^{k}(t,x) \right \|  < \varepsilon,\label{kF} \\
	\left \| l^{k}_{i,s}(t, x^{k}_{i,s}) - l^{k}(t,x) \right \|  < \varepsilon,\label{kl}\\
	\left \| V^{k+1}_{i,s}(t, x^{k}_{i,s}) - V^{k+1}(t,x) \right \|  < \varepsilon,\label{kV}
\end{align}
hold for all $s>W$ and $(t,x)\in[0,T]\times \mathbb{R}^{nN}$,
where $F^{k}(t,x)$, $l^{k}(t,x)$ and $V^{k+1}(t,x) $ are as shown in Lemma \ref{lemma2}, and $x^{k}(t)$ is the corresponding state under the controller \eqref{uk}.
 The proof is divided into two steps.  
 
	The first step is to prove that the following equations hold for any $i$ and $j$,
\begin{align}
	\left \| x^{k}_{i,s}(t) - x^{k}_{j,s}(t) \right \| < \varepsilon, \label{kxij}\\
	\left \| F^{k}_{i,s}(t, x^{k}_{i,s}) - F^{k}_{j,s}(t, x^{k}_{j,s}) \right \|  < \varepsilon, \label{kFij} \\
	\left \| l^{k}_{i,s}(t, x^{k}_{i,s}) - l^{k}_{j,s}(t, x^{k}_{j,s}) \right \|  < \varepsilon,\label{klij}\\
	\left \| V^{k+1}_{i,s}(t, x^{k}_{i,s}) - V^{k+1}_{j,s}(t, x^{k}_{j,s}) \right \|  < \varepsilon. \label{kVij}
\end{align}
To this end, we define the augmented vectors
\begin{align*}
	\mathbf{x}^{k}_{s}(t) &= \begin{bmatrix}
		x^{k}_{1,s}(t)      \\
		\vdots              \\
		x^{k}_{i,s}(t)      \\
		\vdots              \\
		x^{k}_{N,s}(t)
	\end{bmatrix}, &
	\mathbf{\hat{x}}^{k}_{s}(t) &= \begin{bmatrix}
		\hat{x}^{k}_{1,s}(t)    \\
		\vdots                  \\
		\hat{x}^{k}_{i,s}(t)    \\
		\vdots                  \\
		\hat{x}^{k}_{N,s}(t)
	\end{bmatrix},
	\\
%\end{align*}
%\begin{align*}
	%
	\mathbf{F}^{k}_{s}(t) &= \begin{bmatrix}
		F^{k}_{1,s}(t, x^{k}_{1,s})    \\
		\vdots              \\
		F^{k}_{i,s}(t, x^{k}_{i,s})    \\
		\vdots              \\
		F^{k}_{N,s}(t, x^{k}_{N,s})
	\end{bmatrix}, &
	\mathbf{\hat{F}}^{k}_{s}(t) &= \begin{bmatrix}
		\hat{F}^{k}_{1,s}(t, x^{k}_{1,s})  \\
		\vdots                  \\
		\hat{F}^{k}_{i,s}(t, x^{k}_{i,s})  \\
		\vdots                  \\
		\hat{F}^{k}_{N,s}(t, x^{k}_{N,s})
	\end{bmatrix}, \\
	\mathbf{l}^{k}_{s}(t) &= \begin{bmatrix}
		l^{k}_{1,s}(t, x^{k}_{1,s})    \\
		\vdots              \\
		l^{k}_{i,s}(t, x^{k}_{i,s})    \\
		\vdots              \\
		l^{k}_{N,s}(t, x^{k}_{N,s})
	\end{bmatrix}, &
	\mathbf{\hat{l}}^{k}_{s}(t) &= \begin{bmatrix}
		\hat{l}^{k}_{1,s}(t, x^{k}_{1,s})  \\
		\vdots                  \\
		\hat{l}^{k}_{i,s}(t, x^{k}_{i,s})  \\
		\vdots                  \\
		\hat{l}^{k}_{N,s}(t, x^{k}_{N,s})
	\end{bmatrix},
\end{align*}
with 
\begin{align*}
	\hat{x}^{k}_{i,s}(t)=& \mathfrak{x}_{i0} + \int_{0}^{t}   \mathfrak{f}_i(x^{k}_{i,s-1}(\tau)) + \mathfrak{g}_i(x^{k}_{i,s-1}(\tau)) 
	\\ & \times u^{k}_{i,s-1}(\tau, x^{k}_{i,s-1}) d\tau - x^{k}_{i,s-1}(t),
	\\
	\hat{F}^{k}_{i,s}(t, x^{k}_{i,s}) = & \mathfrak{f}_i(x^{k}_{i,s}) + \mathfrak{g}_i(x^{k}_{i,s})u^{k}_{i,s}(t, x^{k}_{i,s})\\& - F^{k}_{i,s-1}(t, x^{k}_{i,s-1})  ,
		\\
%\end{align*}
%\begin{align*}
	%
	\hat{l}^{k}_{i,s}(t, x^{k}_{i,s}) = & \frac{1}{2}(x^{k}_{i,s}(t))'Qx^{k}_{i,s}(t) + \frac{1}{2}(u^{k}_{i,s}(t, x^{k}_{i,s}))'\\
	&\times Ru^{k}_{i,s}(t, x^{k}_{i,s})- l^{k}_{i,s-1}(t, x^{k}_{i,s-1}).
%	 \\
%	%
%	u^{k}_{i,s}(t, x^{k}_{i,s}) = & u^{k}_{i,s}(x^{k}_{i,s}(t)).
\end{align*}
Combining \eqref{xkis}-\eqref{lkis} with the Laplacian matrix $\mathcal{L}$, we obtain the iterative formulas for the augmented vectors as follows:
\begin{align}
	\mathbf{x}^{k}_{s}(t) =& [ (I_N-\frac{1}{\kappa} \mathcal{L} )\otimes I_{nN}] \mathbf{x}^{k}_{s-1}(t) + \delta_s\mathbf{\hat{x}}^{k}_{s}(t),  \label{xkappa} \\
	\mathbf{F}^{k}_{s}(t) =& [ (I_N-\frac{1}{\kappa} \mathcal{L} )\otimes I_{nN}] \mathbf{F}^{k}_{s-1}(t) + \delta_s\mathbf{\hat{F}}^{k}_{s}(t), \label{Fkappa} \\
	\mathbf{l}^{k}_{s}(t) =& [ (I_N-\frac{1}{\kappa} \mathcal{L} )\otimes I_{nN}] \mathbf{l}^{k}_{s-1}(t) + \delta_s\mathbf{\hat{l}}^{k}_{s}(t). \label{lkappa}
\end{align}
We proceed to analyze the boundedness of the above augmented vectors in the sequel.
	To this end, we denote 
\begin{align*}
	\tilde{x}^{k}_{s}(t) = (Z'\otimes I_{nN}) \mathbf{x}^{k}_{s}(t),\\
	\tilde{F}^{k}_{s}(t) = (Z'\otimes I_{nN}) \mathbf{F}^{k}_{s}(t),\\
	\tilde{l}^{k}_{s}(t) = (Z'\otimes I_{nN}) \mathbf{l}^{k}_{s}(t),
\end{align*}
where  
$Z$ is the orthogonal matrix satisfying $Z'\mathcal{L}Z=diag\{ \lambda_1,\ldots,\lambda_N\}$ with $0=\lambda_1 \leq \lambda_2 \leq \ldots \leq \lambda_N $ being the eigenvalues of $\mathcal{L}$.
We then derive the corresponding iterative dynamics as:
\begin{align*} 
	\tilde{x}^{k}_{s}(t) = & [(I_N - \frac{1}{\kappa}Z'\mathcal{L}Z)  \otimes   I_{nN}]\tilde{x}^{k}_{s-1}(t) \\
	&-  \delta_s(Z'  \otimes   I_{nN})\mathbf{\hat{x}}^{k}_{s}(t),\\
	\tilde{F}^{k}_{s}(t) = & [(I_N - \frac{1}{\kappa}Z'\mathcal{L}Z)  \otimes   I_{nN}]\tilde{F}^{k}_{s-1}(t)  \\
	&-  \delta_s(Z'  \otimes   I_{nN})\mathbf{\hat{F}}^{k}_{s}(t),
	\\
	\tilde{l}^{k}_{s}(t) = & [(I_N - \frac{1}{\kappa}Z'\mathcal{L}Z)  \otimes   I_{nN}]\tilde{l}^{k}_{s-1}(t)   \\
	&-  \delta_s(Z'  \otimes  I_{nN})\mathbf{\hat{l}}^{k}_{s}(t).
\end{align*}
By using the properties of Laplacian matrices, we can obtain that $	\tilde{x}^{k}_{s}(t)$, $\tilde{F}^{k}_{s}(t)$ and $	\tilde{l}^{k}_{s}(t) $ are bounded, which in turn ensures the boundedness of the augmented vectors
$\mathbf{x}^{k}_{s}(t)$, $\mathbf{F}^{k}_{s}(t)$, $\mathbf{l}^{k}_{s}(t)$, $	\mathbf{\hat{x}}^{k}_{s}(t) $, $\mathbf{\hat{F}}^{k}_{s}(t) $, and $	\mathbf{\hat{l}}^{k}_{s}(t) $ \cite{Dissolving}.

Then, we define the deviation vectors
\begin{align*}
	\sigma^k_s(t) = [(I_N-Y)\otimes I_{nN}]\mathbf{x}^{k}_{s}(t),\\
	\beta^k_s(t) = [(I_N-Y)\otimes I_{nN}]\mathbf{F}^{k}_{s}(t),\\
	\gamma^k_s(t) = [(I_N-Y)\otimes I_{nN}]\mathbf{l}^{k}_{s}(t),
\end{align*}
where $Y = \frac{1}{N}1_N1_N^{'}$.
Substituting \eqref{xkappa}-\eqref{lkappa} into the defined deviation vectors, and using the matrix property $GY=YG=Y^2=Y$ with $G=I_N-\frac{1}{\kappa} \mathcal{L}$, we obtain the iterative dynamics for the deviation vectors: 
\begin{align*}
	\sigma^k_s(t)=&[(G - Y)  \otimes  I_{nN}]\sigma^k_{s-1}(t) \\
	& +  \delta_s  [(I_N - Y)  \otimes  I_{nN}] \mathbf{\hat{x}}^{k}_{s}(t),\\
	\beta^k_s(t)=&[(G - Y) \otimes  I_{nN}]\beta^k_{s-1}(t) \\
	& +  \delta_s  [(I_N - Y) \otimes  I_{nN}] \mathbf{\hat{F}}^{k}_{s}(t),
	\\
%\end{align*}
%\begin{align*}
	\gamma^k_s(t)=&[(G - Y) \otimes  I_{nN}]\gamma^k_{s-1}(t)  \\
	&+  \delta_s  [(I_N - Y) \otimes  I_{nN}] \mathbf{\hat{l}}^{k}_{s}(t).
\end{align*}
By carrying out the iterative recursion, it yields that
\begin{align*}
	\sigma^k_s(t) =& [(G-Y)\otimes I_{nN}]^s \sigma^k_0(t) + \sum_{\xi=1 }^{s}\{\delta_\xi [(G-Y)\\
	& \otimes I_{nN}]^{s-\xi} 
	[(I_N-Y)\otimes I_{nN}]\mathbf{\hat{x}}^{k}_{\xi-1}(t) \},
%	\\
\end{align*}
\begin{align*}
	\beta^k_s(t) =& [(G-Y)\otimes I_{nN}]^s \beta^k_0(t) + \sum_{\xi=1 }^{s}\{\delta_\xi [(G-Y)
	\\
	& \otimes I_{nN}]^{s-\xi} 
	[(I_N-Y)\otimes I_{nN}]\mathbf{\hat{F}}^{k}_{\xi-1}(t) \},
			\\
%\end{align*}
%\begin{align*}
	%
	\gamma^k_s(t) =& [(G-Y)\otimes I_{nN}]^s \gamma^k_0(t) + \sum_{\xi=1 }^{s}\{\delta_\xi [(G-Y)\\
	&\otimes I_{nN}]^{s-\xi} 
	[(I_N-Y)\otimes I_{nN}]\mathbf{\hat{l}}^{k}_{\xi-1}(t) \}.
\end{align*}
Since the graph is connected, we have  
$ [(G-Y)\otimes I_{nN}]^s \leq c \rho^s $, where $c \ge 0 $ and $\rho \in (0,1)$.
Combined with the boundedness of $\left \|  \sigma^k_0(t)   \right \|$, $\left \|  \beta^k_0(t)   \right \|$, $\left \| \gamma^k_0(t)    \right \|$,  $\left \| \mathbf{\hat{x}}^{k}_{\xi-1}(t)    \right \|$,  $\left \|   \mathbf{\hat{F}}^{k}_{\xi-1}(t)  \right \|$,  $\left \|  \mathbf{\hat{l}}^{k}_{\xi-1}(t)   \right \|$ and $\left \|  [(I_N-Y)\otimes I_{nN}]   \right \|$, 
we arrive at
\begin{align*}
	\lim_{s \to \infty} \left \| \sigma^k_s(t) \right \| =0,\\ 
	\lim_{s \to \infty} \left \| \beta^k_s(t)  \right \| =0,\\
	\lim_{s \to \infty} \left \|  \gamma^k_s(t) \right \| =0,
\end{align*}
which directly implies that \eqref{kxij}-\eqref{klij} hold.
Then, for any agent $i$ and $j$, $ V^{k+1}_{i,s}(t,x^{k}_{i,s})$ and $ V^{k+1}_{j,s}(t,x^{k}_{j,s})$ respectively satisfy the following HJB equations:
\begin{align*} 
	\frac{\partial V^{k+1}_{i,s}(t,x^{k}_{i,s})}{\partial t} + [\frac{\partial V^{k+1}_{i,s}(t,x^{k}_{i,s})}{\partial x^{k}_{i,s}} ]'F^{k}_{i,s}(t,x^{k}_{i,s})\\
	+l^{k}_{i,s}(t,x^{k}_{i,s})=0,\\
	\frac{\partial V^{k+1}_{j,s}(t,x^{k}_{j,s})}{\partial t} +[\frac{\partial V^{k+1}_{j,s}(t,x^{k}_{j,s})}{\partial x^{k}_{j,s}} ]'F^{k}_{j,s}(t,x^{k}_{j,s}) \\+l^{k}_{j,s}(t,x^{k}_{j,s})=0.
\end{align*}
By virtue of the continuous dependence of solutions to first-order linear partial differential equations on their coefficients and terminal conditions, together with the uniform Lipschitz continuity of the system functions and the fact \eqref{kxij}-\eqref{klij}, we further conclude that \eqref{kVij} hold.

The second step is to prove the following convergence results:
\begin{align}
	\lim_{s \to \infty} \left \| \bar{x}^k_s(t) - x^{k}(t)  \right \| =0,\label{xN} 
	\\
	\lim_{s \to \infty} \left \|  \bar{F}^k_s(t) - F^{k}(t,x)  \right \| =0,\label{FN} \\
	\lim_{s \to \infty} \left \| \bar{l}^k_s(t) - l^{k}(t,x)  \right \| =0,\label{lN} \\
	\lim_{s \to \infty} \left \| \bar{V}^{k+1}_s(t) - V^{k+1}(t,x)  \right \| =0,   \label{VN} 
\end{align}
where	
\begin{align*}
	&\bar{x}^k_s(t) = \frac{1}{N}\sum_{i=1}^{N}  x^{k}_{i,s}(t), \ \ \ \ \
	\bar{F}^k_s(t) =\frac{1}{N}\sum_{i=1}^{N}  F^{k}_{i,s}(t,x^{k}_{i,s}),
%\\
\end{align*}
\begin{align*}
	&\bar{l}^k_s(t) = \frac{1}{N}\sum_{i=1}^{N}  l^{k}_{i,s}(t,x^{k}_{i,s}), 	
	\bar{V}^{k+1}_s(t) = \frac{1}{N}\sum_{i=1}^{N}  V^{k+1}_{i,s}(t,x^{k}_{i,s}).	
\end{align*}
To this end, we introduce the iterative equations
\begin{align*}
	{x}^k_s(t) = &{x}^k_{s-1}(t) +\delta_s[x_0 +  \int_{0}^{t}  f(x^{k}_{s-1}(\tau)) + g(x^{k}_{s-1}(\tau))
	\\
	&\times  u^{k}_{s-1}(\tau,x^{k}_{s-1})  d\tau - {x}^k_{s-1}(t) ],
	\\
%	\end{align*}
%	\begin{align*}
	F^{k}_{s}(t) =& F^{k}_{s-1}(t) + \delta_s[ f(x^{k}_{s}(t)) + g(x^{k}_{s}(t))u^{k}_{s}(t,x^{k}_{s}) 
	\\
	&- F^{k}_{s-1}(t) ],
		\\
%\end{align*}
%\begin{align*}
	l^{k}_{s}(t) =& l^{k}_{s-1}(t) +  \delta_s [ \frac{1}{2}(x^{k}_{s}(t))'Qx^{k}_{s}(t)+ \frac{1}{2}(u^{k}_{s}(t,x^{k}_{s}))'\\
	&\times Ru^{k}_{s}(t,x^{k}_{s})-l^{k}_{s-1}(t)],
\end{align*}
and the deviation terms $\Delta{x}^k_s(t) =\bar{x}^k_s(t)-{x}^k_s(t) $, $\Delta{F}^k_s(t) =\bar{F}^k_s(t)-{F}^k_s(t) $,  $\Delta{l}^k_s(t) =\bar{l}^k_s(t)-{l}^k_s(t) $.
Then, we derive the recursive dynamics for these deviation terms as:
\begin{align*}
	\Delta{x}^k_s(t) = (1-\delta_s)\Delta{x}^k_{s-1}(t),\\
	\Delta{F}^k_s(t) = (1-\delta_s)\Delta{F}^k_{s-1}(t),\\
	\Delta{l}^k_s(t) = (1-\delta_s)\Delta{l}^k_{s-1}(t).	
\end{align*}
Under the conditions $0< \delta_s < 1$ and $\lim_{s \to \infty} \delta_s =0$, it follows directly that the deviation sequences converge to zero, i.e.,
\begin{align*}
	\lim_{s \to \infty} \left \|\Delta{x}^k_s(t)  \right \| =0,\\
	\lim_{s \to \infty} \left \| \Delta{F}^k_s(t)  \right \| =0,\\
	\lim_{s \to \infty} \left \|\Delta{l}^k_s(t)   \right \| =0.
\end{align*}
Combining the convergence of the deviation sequences with \eqref{kxij}--\eqref{klij}, we further obtain that
\begin{align}
	\lim_{s \to \infty} \left \| x^{k}_{i,s}(t) - x^{k}_s(t)  \right \| =0,\label{xskconvege} \\ 
	\lim_{s \to \infty} \left \| F^{k}_{i,s}(t,x^{k}_{i,s}) - F^{k}_s(t)  \right \| =0,\label{Fskconvege}\\
	\lim_{s \to \infty} \left \| l^{k}_{i,s}(t,x^{k}_{i,s}) - l^{k}_s(t)  \right \| =0.\label{lskconvege}
\end{align}
Then, similar to the proof in the first step, we consider two HJB equations: one is equation \eqref{Vkis}, whose solution is given by $ V^{k+1}_{i,s}(t,x^{k}_{i,s}) $, and the other is the HJB equation
\begin{align*} 
	\frac{\partial V^{k+1}_{s}(t,x^{k}_{s})}{\partial t} + [\frac{\partial V^{k+1}_{s}(t,x^{k}_{s})}{\partial x^{k}_{s}} ]'F^{k}_{s}(t) + l^{k}_{s}(t)=0.
\end{align*} 
Based on the continuous dependence of solutions to first-order linear PDEs on their coefficients and terminal conditions, combined with the uniform Lipschitz continuity of the system dynamics and the convergence results \eqref{xskconvege}–\eqref{lskconvege}, we conclude that
\begin{align*}
	\lim_{s \to \infty} \left \| V^{k+1}_{i,s}(t,x^{k}_{i,s}) - V^{k+1}_s(t)  \right \| =0.
\end{align*}
%where $ V^{k+1}_{i,s}(t,x^{k}_{i,s}) $ is the solution to HJB \eqref{Vkis}.

In addition, using the convergence properties
\begin{align*}
	\lim_{s \to \infty} \left \| x^{k}_{s}(t) - x^{k}(t)  \right \| =0,\\
	\lim_{s \to \infty} \left \| F^{k}_{s}(t) - F^{k}(t,x)  \right \| =0,
	\\
%\end{align*}
%\begin{align*}
	\lim_{s \to \infty} \left \| l^{k}_{s}(t) - l^{k}(t,x)  \right \| =0,
	\\
%\end{align*}
%\begin{align*}
	\lim_{s \to \infty} \left \| V^{k+1}_{s}(t) - V^{k+1}(t,x)  \right \| =0,
\end{align*}
which follow from Theorem 3.1.1 in \cite{Them3.1.1}, we deduce that \eqref{xN}-\eqref{VN} hold.

Finally, by combining the above result with the convergence of the value iteration established in Theorem 1 of \cite{VI}, we arrive at the desired convergence \eqref{conx}-\eqref{conV}. 
\hfill $\blacksquare$

\bibliographystyle{automatica}

\end{document}